\documentclass[12pt]{amsart}
\usepackage[all]{xy}
\usepackage{fancyhdr}

\usepackage{verbatim}
\usepackage{amssymb, amsmath}
\usepackage{graphicx}

\usepackage{caption}
\usepackage{subcaption}
\usepackage{mathrsfs}
\usepackage[colorlinks,linkcolor=blue,anchorcolor=blue,citecolor=blue,backref=page]{hyperref}

\usepackage{xcolor}

\usepackage{url}

\usepackage{mathrsfs}
\usepackage{bibentry}
\usepackage[norefs,nocites]{refcheck}

\begin{document}
\newtheorem{thm}{Theorem}[section]
\newtheorem{lem}[thm]{Lemma}
\newtheorem{prop}[thm]{Proposition}
\newtheorem{cor}[thm]{Corollary}
\newtheorem{conj}[thm]{Conjecture}
\newtheorem{hypothesis}[thm]{Hypothesis}
\newtheorem{proj}[thm]{Project}
\newtheorem*{remark}{Remark}
\newtheorem{question}[thm]{Question}


\makeatletter
\@namedef{subjclassname@2020}{\textup{2020} Mathematics Subject Classification}
\makeatother

\newcommand{\rad}{\operatorname{rad}}

\newcommand{\Z}{{\mathbb Z}} 
\newcommand{\Q}{{\mathbb Q}}
\newcommand{\R}{{\mathbb R}}
\newcommand{\C}{{\mathbb C}}
\newcommand{\N}{{\mathbb N}}
\newcommand{\FF}{{\mathbb F}}
\newcommand{\fq}{\mathbb{F}_q}
\newcommand{\rmk}[1]{\footnote{{\bf Comment:} #1}}

\renewcommand{\mod}{\;\operatorname{mod}}
\newcommand{\ord}{\operatorname{ord}}
\newcommand{\TT}{\mathbb{T}}
\renewcommand{\i}{{\mathrm{i}}}
\renewcommand{\d}{{\mathrm{d}}}
\renewcommand{\^}{\widehat}
\newcommand{\HH}{\mathbb H}
\newcommand{\Vol}{\operatorname{vol}}
\newcommand{\area}{\operatorname{area}}
\newcommand{\tr}{\operatorname{tr}}
\newcommand{\norm}{\mathcal N} 
\newcommand{\intinf}{\int_{-\infty}^\infty}
\newcommand{\ave}[1]{\left\langle#1\right\rangle} 
\newcommand{\Var}{\operatorname{Var}}
\newcommand{\Prob}{\operatorname{Prob}}
\newcommand{\sym}{\operatorname{Sym}}
\newcommand{\disc}{\operatorname{disc}}
\newcommand{\CA}{{\mathcal C}_A}
\newcommand{\cond}{\operatorname{cond}} 
\newcommand{\lcm}{\operatorname{lcm}}
\newcommand{\Kl}{\operatorname{Kl}} 
\newcommand{\leg}[2]{\left( \frac{#1}{#2} \right)}  
\newcommand{\Li}{\operatorname{Li}}

\newcommand{\sumstar}{\sideset \and^{*} \to \sum}

\newcommand{\LL}{\mathcal L} 
\newcommand{\sumf}{\sum^\flat}
\newcommand{\Hgev}{\mathcal H_{2g+2,q}}
\newcommand{\USp}{\operatorname{USp}}
\newcommand{\conv}{*}
\newcommand{\dist} {\operatorname{dist}}
\newcommand{\CF}{c_0} 
\newcommand{\kerp}{\mathcal K}

\newcommand{\Cov}{\operatorname{cov}}
\newcommand{\Sym}{\operatorname{Sym}}

\newcommand{\Ht}{\operatorname{Ht}}

\newcommand{\E}{\operatorname{\mathbb E}} 
\newcommand{\sign}{\operatorname{sign}} 
\newcommand{\meas}{\operatorname{meas}} 

\newcommand{\divid}{d} 

\newcommand{\GL}{\operatorname{GL}}
\newcommand{\SL}{\operatorname{SL}}
\newcommand{\re}{\operatorname{Re}}
\newcommand{\im}{\operatorname{Im}}
\newcommand{\res}{\operatorname{Res}}

 \newcommand{\EWp}{\mathbb E^{\rm WP}_g} 
\newcommand{\orb}{\operatorname{Orb}}
\newcommand{\supp}{\operatorname{Supp}}
\newcommand{\mmfactor }{\textcolor{red}{c_{\rm Mir}}}
\newcommand{\Mg}{\mathcal M_g} 
\newcommand{\MCG}{\operatorname{Mod}} 
\newcommand{\Diff}{\operatorname{Diff}} 
\newcommand{\If}{I_f(L,\tau)}
\newcommand{\GOE}{\operatorname{GOE}}
\newcommand{\Ex}{\mathcal E} 
\newcommand{\alow}{\mathbf a}
\newcommand{\adag}{\alow^\dagger} 
\newcommand{\spec}{\mathcal X}
\newcommand{\eigen}{E} 
\newcommand{\tridiag}{\operatorname{tridiag}}

\title[The density conjecture for Juddian points]{The density conjecture for Juddian points  for the quantum Rabi model  }
\date{\today}
\author{Rishi Kumar and Ze\'ev Rudnick}
\address{School of Mathematical Sciences, Tel Aviv University, Tel Aviv 69978, Israel}
\email{rkumar@tauex.tau.ac.il} \email{rudnick@tauex.tau.ac.il}

\begin{abstract}
We study doubly degenerate (Juddian) eigenvalues for the Quantum Rabi Hamiltonian, a simple model of the interaction between a two-level atom and a single quantized mode of light. We prove a strong form of the density conjecture of Kimoto, Reyes-Bustos, and Wakayama, showing that any fixed value of the splitting between the two atomic levels, there is a dense set of coupling strengths for which the corresponding Rabi Hamiltonian admits Juddian eigenvalues. We also construct infinitely many sets of parameters for which the Rabi Hamiltonian admits two distinct Juddian eigenvalues. 
The fine structure of the zeros of classical Laguerre polynomials plays a key role in our methods. 
\end{abstract}
\thanks{This research was supported by the ISRAEL SCIENCE FOUNDATION (grant No. 2860/24)}
 \maketitle

\section{Introduction}
The quantum Rabi model (QRM)
describes the interaction between a two-level atom (qubit), coupled to a quantized, single-mode harmonic oscillator. The QRM  has its origin in the semi-classical model of such interactions due to I. I. Rabi (1936) 
and its fully quantized version due to Jaynes and Cummings (1963), who introduced a ``rotating wave approximation'' which is exactly solvable and is very useful for small coupling.  However, recent advances in experimental techniques 
bring to the fore the need to understand the full QRM, without imposing the rotating wave approximation \cite{Larson&Mavrogordatos}. 
Thus the need for a better understanding of the mathematical theory of the QRM. In this note we advance this understanding by focusing on degenerate points in the spectrum, the ``Juddian eigenvalues''.

The Hamiltonian of the QRM is, after some simplifications,  
\[
H_{g,\Delta}= \adag \alow + \Delta \sigma_z  + g\sigma_x (\alow + \adag)    
\]
where $\sigma_x=\left(\begin{smallmatrix}0&1\\1&0 \end{smallmatrix}\right)$, $\sigma_z=\left(\begin{smallmatrix}1&0\\0&-1 \end{smallmatrix}\right)$ are the Pauli matrices of the two-level atom, $\Delta>0$ is half the splitting between the two atomic levels, $\adag$ and $\alow$ are the creation and annihilation operators of the harmonic oscillator whose frequency has been set to unity, and $g>0$ is the light-matter coupling strength.    

Eigenvalues $\eigen$ for which $\eigen+g^2$ is an integer are called ``exceptional''. These can only exist for special choices of the parameters $g$ and $\Delta$. Degeneracies in the spectrum of $H_{g,\Delta}$ (that is the eigenspace has dimension bigger than one) may only occur for exceptional eigenvalues. In such cases the eigenvalue is called ``degenerate exceptional'', or ``Juddian'', necessarily the dimension of the eigenspace is two, and in a suitable representation the eigenvectors take a simple form.  
 
The values of the parameters where Juddian eigenvalues occur are given by (at least one of) a sequence of polynomial constraints \cite{Kus}
$$P_n( (2g)^2,\Delta^2)=0.$$   
Since we have a countable collection of polynomial constraints, for generic parameters $(g,\Delta)$, there are no Juddian eigenvalues; the set of parameters where there is at least one Juddian eigenvalue has measure zero. 
A conjecture raised in \cite[Conjecture 6.1]{RBW} is that the set of parameters for which there is a Juddian eigenvalue is dense in the $(g,\Delta)$$-$plane. 
We prove a strong form of the density conjecture:  
\begin{thm}\label{thm:densityJudd}
Fix the level splitting $2\Delta>0$. Then, the set of coupling constants $g>0$ for which $H_{g,\Delta}$ admits Juddian eigenvalues is dense in the set 
of all possible coupling constants.  
Moreover, we have a limit density: For any fixed $\Gamma>0$, as $N\to \infty$,
\[
\#\{g \leq \Gamma :  N-g^2 \mathrm{\;  is \; a \; Juddian \; eigenvalue \; for\; } H_{g,\Delta} \} \sim
\frac{4 }{\pi}\cdot  \Gamma \cdot  \sqrt{N }.
\]
\end{thm}

It is believed that for a given parameter pair $(g,\Delta)$, ($g>0$) there are only finitely many Juddian eigenvalues, though this is not proven. Until this time, there were no known examples of parameters where there was more than one Juddian eigenvalue. 
We show that these do exist:
 
 
\begin{thm}\label{thm:Juddv1m}
Fix $m\geq 1$. There are infinitely many $N$'s for which there are parameters  $(g_N,\Delta_N)$ (all distinct) where the eigenvalue spectrum of the Rabi Hamiltonian contains the two Juddian eigenvalues $m-g_N^2$ and $N-g_N^2$.
\end{thm}

For instance, the first $n$ for which there is some $(g,\Delta)$ with both $1-g^2$ and $n-g^2$ as Juddian eigenvalues is $n=8$, and numerics indicate that all $8\leq n\leq 30$ have this property.
 
While we showed that there are infinitely many distinct parameters $(g,\Delta)$ having two Juddian eigenvalues, it seems unlikely that there is any $(g,\Delta)$ having three such eigenvalues.

\subsection{About the proofs}
 
A key ingredient in both Theorems~\ref{thm:densityJudd} and \ref{thm:Juddv1m} is that the restriction of the constraint polynomial $P_N(x,y)$ to the $x$-axis, that is the  un-physical case when the level splitting is zero,  is given by 
 \[
 P_N(x,0) = (-1)^N(N!)^2 L_N(x)
 \] 
where $L_N(x)=\sum_{m=0}^N \binom{N}{m}\frac{(-x)^m}{m!}$ is the classical Laguerre polynomial \cite{LFTB, KRBW}.   
When $y>0$, we use a determinantal representation due to Kimoto, Reyes-Bustos and Wakayama \cite{KRBW} to write $P_N(x,y)$ in terms of the characteristic polynomial of a certain  symmetric tri-diagonal matrix, and then apply perturbation theory to show that for $N$ sufficiently large, and for $m=m(N)$ tending to infinity with $N$ while $m=o(\sqrt{N})$, the zeros $\alpha_1(y) <\alpha_2(y) < \dots< \alpha_N(y)$ of $x\mapsto P_N(x,y)$ satisfy
\[
  \lambda_{N-m,k-m}-\frac{y}{m+1}\leq  \alpha_k(y)  < \lambda_{N-m,k},\quad k= m+1,\dots , N-m
\]
where $\lambda_{N-m,1}<\dots <\lambda_{N-m,N-m}$ are the zeros of the Laguerre polynomial $L_{N-m}(x)$.  
We then appeal to the known distribution of ``small'' zeros of $L_n(x)$, namely that for fixed $z>0$ and $n\to  \infty$,  
there are asymptotically $\frac {2}{\pi} \sqrt{nz}$ zeros $\lambda_{n,k}\leq z$ \cite[Theorem 2]{Gawronski}, to deduce the density conjecture in the form 
of Theorem~\ref{thm:densityJudd}. 
These results about the distribution of zeros of the Laguerre polynomial are also instrumental in finding two Juddian eigenvalues for 
Theorem~\ref{thm:Juddv1m}.

\section{PRELIMINARY RESULTS}
This section will introduce constraint polynomials, tri-diagonal matrix, and results for the proofs in Sections \ref{section for theorem 1} and \ref{section for theorem 2}.
\subsection{Background on constraint polynomials}
Eigenvalues $\eigen$ where $z=\eigen+g^2$ are not an integer are called the ``regular'' spectrum. Braak~\cite{Braak} showed that the regular part of the spectrum of the Rabi model is given in terms of the zeros of $G(z)=G_+(z)G_-(z)$ where $z = \eigen+g^2$ is the shifted spectral parameter, and
\begin{equation*}
G_\pm(z) = \sum_{n=0}^\infty K_n(z;g,\Delta) g^n \left( 1 \mp \frac{\Delta}{z-n} \right) ,
\end{equation*}
where the functions $K_n(z;g,\Delta)$ satisfy a recursion
\begin{equation*}
\begin{split}
nK_n(z;g,\Delta) &= f_{n-1}(z;g,\Delta)K_{n-1}(z;g,\Delta)-K_{n-2}(z;g,\Delta),
\\
f_m(z) &= 2g+\frac 1{2g}\left( \frac{\Delta^2}{z-m}+m-z \right),
\end{split}
\end{equation*}
with initial conditions $K_0\equiv 1$, $K_1 = f_0(z)$.

The Juddian eigenvalues are of the form $\eigen=n-g^2$, where $n$ is an integer satisfying
$$K_n(n;g,\Delta) =0.$$
In this case the possible poles of $G_\pm(z)$ at $z=n$ are cancelled, as in Figure~\ref{figsub:GfunctionsJudd}, and the spectrum is doubly degenerate, that is the corresponding eigenspace has dimension two.
\begin{figure}[ht]
\begin{subfigure}[b]{0.45\textwidth}
\includegraphics[height=50mm , width=\textwidth]{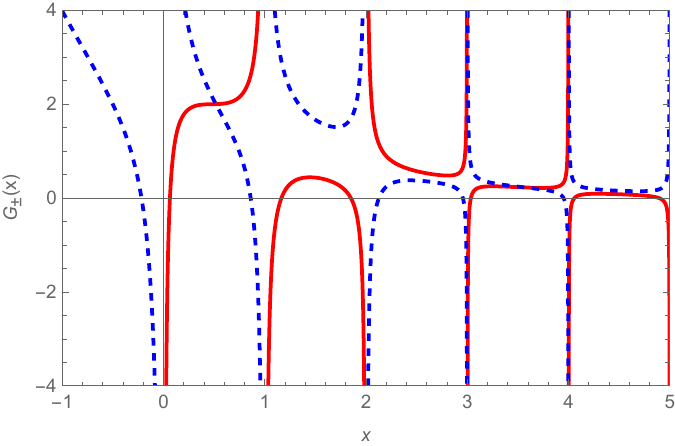}
\caption{   $g=0.7$ and $\Delta=0.4$ }\label{figsub:GBraak}
\end{subfigure}
\hfill
\begin{subfigure}[b]{0.45\textwidth}
\includegraphics[height=50mm, width=\textwidth ]{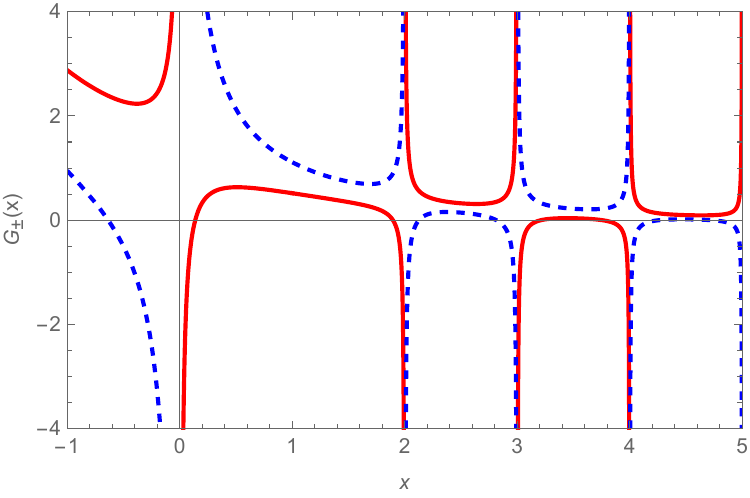}
\caption{ $g=1/ \sqrt{8}$ and  $\Delta=1/\sqrt{2}$.   }   \label{figsub:GfunctionsJudd}
\end{subfigure}
\caption{The functions $G^+(z)$ (red) and $G^-(z)$ (blue, dashed) for $g=0.7$ and $\Delta=0.4$ (\ref{figsub:GBraak}),
and for  $g=1/ \sqrt{8}$, $\Delta=1/\sqrt{2}$ (\ref{figsub:GfunctionsJudd}), when the pole at $x=1$ is cancelled and $\eigen= 1-g^2$ is an exceptional (Juddian) eigenvalue.  }
\label{fig:Gfunctions}
\end{figure}


We write the condition $K_n(n;g,\Delta) =0$ in the equivalent form
$$ P_n( (2g)^2,\Delta^2)=0,$$
where
\[
P_n( (2g)^2,\Delta^2):= (n!)^2 (2g)^n K_n(n;g,\Delta)
\]
is a polynomial in $X=(2g)^2$ and $Y=\Delta^2$, which turns out to have degree $n$. 
These constraint polynomials have received a lot of attention recently \cite{Kus, LFTB, KRBW, RBW} and we will need several of their properties.

The first few are
\[
P_1(X,Y) = X+Y-1, \qquad P_2(X,Y)=2 X^2+3 X Y+Y^2-8 X-5 Y+4,
\]
\begin{multline*}
P_3(X,Y)=6 X^3+11 X^2 Y+6 X Y^2 +Y^3
\\ -54 X^2-58 X Y-14 Y^2  +108 X  +49 Y-36,
\end{multline*}
\begin{multline*}
P_4(X,Y)=24 X^4+50 X^3 Y +35 X^2 Y^2+10 X Y^3 +Y^4
\\
-384 X^3-542 X^2 Y -230 X Y^2-30 Y^3
\\+1728 X^2     +1444 X Y+273 Y^2 -2304 X-820 Y+576.
\end{multline*}

There is a recursive procedure to calculate these polynomials, due to Ku\'{s} \cite{Kus}.
Define polynomials $P^{(n)}_k(X,Y)$, $k=0,1,\dots, n$, by the recursion
\begin{multline}\label{Kus recurrence}
P^{(n)}_k(X,Y) = (kX+Y-k^2)\cdot  P^{(n)}_{k-1}(X,Y)
\\
-k(k-1)(n-k+1) \cdot X \cdot P^{(n)}_{k-2}(X,Y)
\end{multline}
for $  2\leq k\leq n$, with initial conditions
\[
P_0^{(n)}\equiv 1, \quad  P_1^{(n)}(X,Y)=X+Y-1.
\]
Then Ku\'{s} shows that
 $$P_n(X,Y) = P^{(n)}_n(X,Y)  .$$

From the recursion \eqref{Kus recurrence} we see that $P^{(n)}_k(X,Y)$ has integer coefficients,  has degree $k$,  and that the coefficient of  $X^k$ is $k!$, and of $Y^k$ is $1$.

\begin{lem}\label{lem:Intersect with Y axis}
We have
$$
P_N(0,Y) = \prod_{m=1}^N(Y-m^2).
$$
Hence the solutions of  $P_N(0,Y)=0$ are $Y=1,2^2,\dots,N^2$.
 \end{lem}
\begin{proof}
We use Ku\'{s}'s recursion \eqref{Kus recurrence} specialized to $X=0$, which gives
\[
P^{(n)}_k(0,Y) = (Y-k^2)\cdot  P^{(n)}_{k-1}(0,Y)
\]
with initial conditions $P_0^{(n)}(0,Y)\equiv 1$,  $P_1^{(n)}(0,Y)=Y-1$ and hence by induction on $k$ we obtain
\[
P_n(0,Y)=P^{(n)}_n(0,Y) = \prod_{k=1}^n(Y-k^2) .
\]
\end{proof}

\begin{lem}\label{prop:Laguerre}
The restriction of the constraint polynomial to the $x$-axis is a multiple of the classic Laguerre polynomial:
\[
  P_N(x,0)    = (-1)^N (N!)^2 L_N(x).
\]
\end{lem}

Recall the Laguerre polynomials are given by
\begin{equation*}
L_N(x)=\sum_{m=0}^N  (-1)^m \binom{N}{m} \frac{x^m}{m!}  .
\end{equation*}
The relation of $P_N(x,0)$ with the Laguerre polynomials appears in \cite[\S 3.2, eq (13)]{LFTB},  
a proof is given in \cite[Theorem 3.18]{KRBW}. Lemma~\ref{lem:Intersect with Y axis} is given in the course of the proof of Lemma 3.10 of \cite{KRBW}.

The Laguerre polynomial $L_N(x)$ has $N$ zeros, which are all real and positive, denoted by
 $$0<\lambda_{N,1}<\lambda_{N,2}<\dots<\lambda_{N,N}.$$
The distribution of the zeros of Laguerre polynomials plays a significant role in the proofs of Theorems  \ref{thm:densityJudd} and \ref{thm:Juddv1m}. The following theorem is due to Gawronski:
\begin{thm}\emph{\cite[Theorem 2]{Gawronski}}\label{Gawronski thm}
Let $x>0$. Then
\begin{equation*}
 \lim_{N\to \infty}\frac{1}{\sqrt{N}}\cdot \#\{k=1,\ldots, N\, : \lambda_{N,k}\leq x\}= \frac{2}{\pi}\cdot \sqrt{x}.   
\end{equation*}
\end{thm}
\subsection{A matrix representation}

Kimoto, Reyes-Bustos and Wakayama \cite{KRBW} gave a representation of the constraint polynomials as the characteristic polynomial of a symmetric tri-diagonal matrix which we will use.
Denote by
\[
\tridiag_N\begin{pmatrix} a_i&b_i\\c_i&\end{pmatrix} = \begin{pmatrix} a_1&b_1&0&0&\dots \\c_1&a_2&b_2&0&\dots
\\ 0&c_2&a_3& b_3&0 \\  \vdots
\end{pmatrix}
\]
the tri-diagonal $N\times N$ matrix with diagonal entries $a_1,\dots, a_N$, upper off-diagonal entries $b_1,\dots, b_{N-1}$, and lower off-diagonal entries $c_1,\dots$,  $c_{N-1}$. We note that
\begin{equation}\label{det replace offdiag}
\det \tridiag_N\begin{pmatrix} a_i&b_i\\c_i&\end{pmatrix} = \det \tridiag_N\begin{pmatrix} a_i&b'_i\\c'_i&\end{pmatrix}
\end{equation}
whenever $b_ic_i=b'_ic'_i$.

Let
\begin{equation*}
\begin{split}
M_N & =    \tridiag_N \begin{pmatrix} 2(N-i)+1 &N-i \\N-i &\end{pmatrix}
\\
&= \begin{pmatrix} 2N-1&N-1& 0 &  &\\N-1&2N-3& N-2 &0 \\ 0&N-2& 2N-5& N-3 &0&  \\ \vdots
\\ &  & 0 & 1&1
\end{pmatrix}.
\end{split}
\end{equation*}	

\begin{lem}\label{lem:tridiag}
Let  $D_N$ be the diagonal matrix with entries $1,2,\dots, N$. Then
\[
P_N (x,y) = N!\det\left( xI_N -\left(M_N- yD_N^{-1}\right) \right).
\]
\end{lem}
\begin{proof}
Let $S_N$ the tri-diagonal symmetric matrix
\[
S_N=\tridiag_N \begin{pmatrix} -i((2N-i)+1) & (N-i)\sqrt{i(i+1)} \\ (N-i)\sqrt{i(i+1)} & \end{pmatrix}.
\]
Then \cite[Corollary 3.4]{KRBW}
\[
P_N^{(N)}(x,y) = \det\left(y I_N+ x D_N +S_N\right)  .
\]
Removing the matrix $D_N$ which has $\det D_N=N!$, so as to write
\[
\det\left(y I_N+ x D_N +S_N\right)   = N! \det \left(x I_N+y D_N^{-1} + D_N^{-1} S_N  \right) ,
\]
replacing $ D_N^{-1}S_N$ by its conjugate
\[
D_N^{-1/2}S_N D_N^{1/2} = \tridiag_N     \begin{pmatrix}-( 2(N-i)+1) &N-i \\N-i &\end{pmatrix},
\]
and replacing the off-diagonal entries by their negatives, we obtain using \eqref{det replace offdiag}
\[
P_N^{(N)}(x,y) = N!\det\left( xI_N -\left(M_N- yD_N^{-1}\right) \right)
\]
with
\[
M_N =      \tridiag_N \begin{pmatrix} 2(N-i)+1 &N-i \\N-i &\end{pmatrix}
\]
as claimed.
\end{proof}

Thus, for fixed $y$, the zeros of $Q_N(x) = P_N(x,y)/N!$ are the eigenvalues of the symmetric tri-diagonal matrix $M_N- y D_N^{-1}$. 

 \begin{lem} \label{lem:distinct roots of Q}
 Fix $y>0$. Then the polynomial $Q_n(x):=P_n(x,y)/n!$ 
 has exactly $n$ real roots $\alpha_1(y)<\alpha_2(y)<\dots <\alpha_n(y)$, all distinct.
 \end{lem}
 \begin{proof} 
 We saw in Lemma~\ref{lem:tridiag} that we can write $Q_n(x)$ as the characteristic polynomial of the symmetric tri-diagonal matrix $A_n=M_n-y D_n^{-1}$, all of whose immediate off-diagonal entries are non-zero. Every real symmetric matrix has all its eigenvalues real, and if further it is tri-diagonal with all immediate off-diagonal entries being nonzero then the eigenspaces are one-dimensional, hence the eigenvalues are simple; this is a simple fact, see e.g. \cite[\S 0.9.9]{HJ}. We recall the argument: Write 
 $A_n = \tridiag_n   \left( \begin{smallmatrix}a_i &b_i \\ b_i &\end{smallmatrix} \right)$,    
 with $b_i = n-i$  nonzero ($i=1,\dots, n-1$), 
 then the eigenvalue equation $(A_n-\lambda I)\vec x=0$ reads as
 \[
 \begin{split}
 (a_1-\lambda)x_1 + b_1x_2  &=0
 \\ b_1x_1 + (a_2-\lambda)x_2 + b_2x_3 &=0 
 \\
 \dots & 
 \end{split}
 \]
 and since $b_1\neq 0$, we find $x_2 = (\lambda-a_1)x_1/b_1$ is determined by $x_1$, and since $b_2\neq 0$, 
 $ x_3 = -(b_1x_1+(a_2-\lambda)x_2)/b_2$ so that $x_3$ is also determined by $x_1$, etcetera.
 \end{proof} 
 
According to \cite[Theorem 4.3]{KRBW}, there are exactly $N-\lfloor \sqrt{y_0} \rfloor$ positive zeros.

\section{Proof of the density conjecture}\label{section for theorem 2}
According to Lemma~\ref{lem:tridiag}, $P_N(x,y)/N!$ is the  characteristic polynomial of $M_n-y D_n^{-1}$. 
We view $M_N- y D_N^{-1}$ as a perturbation of the real symmetric matrix~$M_N$.
We need a lemma of Weyl on the effect on the eigenvalues of a perturbation of a given real symmetric (or Hermitian) matrix, see e.g. \cite[Theorem 4.3.1]{HJ}. For any Hermitian $N\times $N matrix $H$, denote its eigenvalues in increasing order by
\[
 \lambda_1(H)\leq \lambda_2(H)\leq \dots \leq \lambda_N(H) .
\]

\begin{lem}[Weyl]\label{Weyl's lemma}
Let $M$ and $B$ be Hermitian $N\times N$ matrices. Then
\[
\lambda_1(B)\leq \lambda_i(M+B)-\lambda_i(M)   \leq  \lambda_N(B) .
\]
\end{lem}

Taking $M=M_N$ and $B=-y D_N^{-1}$, we obtain a corollary about the motion of the zeros of the polynomial $x\mapsto P_N(x,y)$ as we vary $y$.

\begin{cor}\label{cor:motion of zeros} 
Fix $y >0$. The zeros $\alpha_1(y)< \alpha_2(y)< \dots< \alpha_N(y)$ of the polynomial
$Q_N(x)=P_N(x,y)/N!$ are decreasing as a function of $y$, and
\[
\alpha_i(y)-\alpha_i(0)\in [-y,-\frac{y}{N}].
\]
\end{cor}

Recall that for $y=0$, the zeros $\alpha_i(0)$ of $P_N(x,0)$ are the zeros  $\lambda_{N,i}$ of the Laguerre polynomial $L_N(x)$. We would like to deduce that the zeros $\alpha_i(y_0)$ of $Q_N(x)=P_N(x,y_0)/N!$ are close to the zeros $\lambda_{N,i}$ of the Laguerre polynomial $L_N(x)$. However, Corollary~\ref{cor:motion of zeros} only lets us deduce that
\[
|  \alpha_i(y_0)-\lambda_{N,i} |\leq y_0 .
\]
In order to show that this difference tends to zero in a relevant range, we add an ingredient.

We recall Cauchy's interlacing theorem, see e.g. \cite[Theorem 4.3.28]{HJ}:
\begin{thm}
Let $A=A^*$ be an $n\times n$ Hermitian matrix, and $B=B^*$ any $(n-m)\times (n-m)$ principal minor, obtained by deleting from $A$ the $i$-th row and $i$-th column, for $m$ different values of $i$. For instance we can take $B$ as the lower $(n-m)\times (n-m)$ submatrix of $A$:
\[
A= \begin{pmatrix} C_{m} & X^*\\ X  & B_{n-m}\end{pmatrix} .
\]
Let $\alpha_1\leq \alpha_2\leq \dots \leq \alpha_n$ be the eigenvalues of $A$ in increasing order, and $\beta_1\leq \beta_2\leq \dots \leq \beta_{n-m}$ those of $B$. Then
\[
\alpha_k\leq \beta_k\leq \alpha_{k+m},\quad k=1,\dots,n-m
\]
and if $m\leq n/2$,
\[
\beta_{k-m} \leq \alpha_k\leq \beta_k,\quad k= m+1,\dots , n-m.
\]
\end{thm}

We apply the interlacing theorem to the matrix $A_N=M_N-y D_N^{-1}$, take $m\leq N/2$ (chosen later) and $B$ the lower $(N-m)\times(N-m)$ principal minor. Therefore the eigenvalues of $A_N$, which are the zeros of $Q_N(x)= P_N(x,y)/N!$, interlace the eigenvalues $\beta_1< \dots < \beta_{N-m}$ of~$B$.

Observe that the lower $(N-m)\times (N-m)$ principal minor of $M_N$ is $M_{N-m}$. Thus
\[
B=M_{N-m}  - y\begin{pmatrix} \frac 1{m+1} & &  \\  & \frac 1{m+2} & \\ \vdots \\  & & \frac 1N \end{pmatrix}
\]
Therefore, the eigenvalues $\beta_1< \dots <  \beta_{N-m}$ of $B$ are perturbations of the eigenvalues of $M_{N-m}$, which are the zeros $\lambda_{N-m,1}<\dots <\lambda_{N-m,N-m}$ of the Laguerre polynomial $L_{N-m}(x)$. Using Lemma~\ref{Weyl's lemma} we obtain:
\[
\beta_i -\lambda_{N-m,i}\in [-\frac{y}{m+1},-\frac{y}{N}],\quad i=1,\dots, N-m
\]
in particular
\begin{equation}\label{beta vs lambda}
-\frac{y}{m+1}\leq \beta_i -\lambda_{N-m,i}<0, \quad i=1,\dots, N-m.
\end{equation}
The interlacing theorem together with \eqref{beta vs lambda} gives that the eigenvalues $\alpha_1(y)<\dots < \alpha_N(y)$ of $A_N$ satisfy
\[
\lambda_{N-m,k-m}-\frac{y}{m+1}\leq   \beta_{k-m} \leq \alpha_k(y)\leq \beta_k < \lambda_{N-m,k}, 
\]
for 
\[
 k= m+1,\dots , N-m.
 \]
Since these eigenvalues are the zeros of the  polynomial $x\mapsto P_N(x,y)$, 
we find

\begin{thm}\label{thm:consequence of interlacing} 
Fix $y>0$.   Let $m\leq N/2$.    Then the zeros $\alpha_1(y)< \dots < \alpha_N(y)$ of the
polynomial $Q_N(x)= P_N(x,y)/N!$ satisfy
\[
  \lambda_{N-m,k-m}-\frac{y}{m+1}\leq     \alpha_k(y)  < \lambda_{N-m,k},\quad k= m+1,\dots , N-m
\]
where $\lambda_{N-m,1}<\dots <\lambda_{N-m,N-m}$ are the zeros of the Laguerre polynomial $L_{N-m}(x)$.
\end{thm}

\subsection{Proof of Theorem~\ref{thm:densityJudd} }

We first show:

\begin{thm}\label{thm: density of alphas}
Fix $y>0$. Then for any $x>0$,
\[
\lim_{N\to \infty} \frac 1{\sqrt{N}}  \#\{ k=1,\dots, N: \alpha_k(y)\leq x\} =\frac {2}{\pi}\sqrt{x} .
\]
\end{thm}

\begin{proof}
Let $n=N-m$. Using Theorem~\ref{thm:consequence of interlacing}, we obtain
\[
\begin{split}
\#\{ k=1,\dots, N: \alpha_k(y)\leq x\} &\geq
\#\{ k=m+1,\dots, n: \alpha_k(y)\leq x\}
\\
 &\geq \#\{k=m+1,\dots,n: \lambda_{n,k}\leq x\}
\\
&\geq \#\{k=1,\dots,n: \lambda_{n,k}\leq x\}  -m.
\end{split}
\]
By Theorem \ref{Gawronski thm}, we have
\[
 \#\{k=1,\dots,n: \lambda_{n,k}\leq x\}\sim \frac{2}{\pi}\sqrt{x}\cdot \sqrt{n}
\]
and since we choose $m=o(\sqrt{N})$ (so that $\sqrt{N}\sim \sqrt{n}$) we obtain
\[
\liminf_{N\to \infty} \frac 1{\sqrt{N}} \#\{ k=1,\dots, N: \alpha_k(y)\leq x\} \geq  \frac{2}{\pi}\sqrt{x}.
\]

Next,  we will see that for any fixed $\varepsilon>0$,
\[
\limsup_{N\to\infty } \frac 1{\sqrt{N}} \#\{ k=1,\dots, N: \alpha_k(y)\leq x\} \leq \frac{2}{\pi} \sqrt{x+\varepsilon}.
\]
Indeed,  from Theorem~\ref{thm:consequence of interlacing}  we find
\[
\begin{split}
\#\{ k=1,\dots, N : & \,\alpha_k(y)\leq x\}
\leq \#\{ k=m+1,\dots, n: \alpha_k(y)\leq x\} +2m
\\
&\leq \#\{ k=m+1,\dots, n: \lambda_{n,k-m}\leq x+\frac{y}{m+1} \} +2m
\\
&= \#\{ i=1,\dots, n-m: \lambda_{n,i}\leq x+\frac{y}{m+1} \} +2m
\\
&\leq \#\{ i=1,\dots, n: \lambda_{n,i}\leq x+\frac{y}{m+1} \} +2m.
\end{split}
\]
Taking  $N\gg 1$   so that $m=m(N)$ is sufficiently large so that $y/(m+1)<\varepsilon$ (but still $m=o(\sqrt{N})$)   we find
\[
\begin{split}
\#\{ k=1,\dots, N: \alpha_k(y)\leq x\}
&  \leq \#\{ i=1,\dots, n: \lambda_{n,i}\leq x +\varepsilon\} +2m.\\
\end{split}
\]
Again applying Gawronski's theorem gives that as $N\to \infty$,
\[
 \#\{ k=1,\dots, n: \lambda_{n,k}\leq x +\varepsilon\} \sim \frac{2}{\pi }\sqrt{n} \sqrt{x+\varepsilon}
\]
and therefore (recall $m=o(\sqrt{N})$ and $\sqrt{N}\sim \sqrt{n}$),
\[
\limsup_{N\to\infty } \frac 1{\sqrt{N}} \#\{ k=1,\dots, N: \alpha_k(y)\leq x\} \leq \frac{2}{\pi} \sqrt{x+\varepsilon}.
\]
Since $\varepsilon>0$ is arbitrary, we obtain
\[
\lim_{N\to \infty} \frac 1{\sqrt{N}}  \#\{ k=1,\dots, N: \alpha_k(y)\leq x\} =\frac {2}{\pi}\sqrt{x} .
\]
\end{proof}

To translate to the precise statement of Theorem~\ref{thm:densityJudd}, recall that for $N-g^2$ to be a Juddian eigenvalue of the Hamiltonian $H_{g,\Delta}$ is equivalent to $g$ satisfying the constraint equation $P_N( (2g)^2,\Delta^2)=0$, meaning $(2g)^2=\alpha_k(\Delta^2)$ for some $k=1,\dots, N$. Thus
\begin{multline*}
\#\{g \leq \Gamma :  N-g^2 \mathrm{\;  is \; a \; Juddian \; eigenvalue \; for\; } H_{g,\Delta} \}
\\
=  \#\{ k=1,\dots, N: \alpha_k(\Delta^2)\leq (2\Gamma)^2\} .
\end{multline*}
 By Theorem~\ref{thm: density of alphas}, as $N\to \infty$, this is asymptotic to
\[
\sqrt{N} \cdot \frac{2}{\pi} \cdot 2 \Gamma
\]
 as claimed. 

In particular, if we fix $x_0>0$ and $0<\delta<x_0$, the the number of coupling constants $g\in (x_0-\delta,x_0+\delta)$ for which $N-g^2$ is a Juddian eigenvalue for $H_{g,\Delta}$ is asymptotically $\sqrt{N} \cdot\frac{8\delta}{\pi}>0$ as $N\to \infty$, so we obtain density.

\section{Proof of Theorem \ref{thm:Juddv1m}}\label{section for theorem 1}
\subsection{Properties of $Z_m$}
 We denote by $Z_n$ the zero locus 
 of the constraint polynomial $P_n(X,Y)$,
  \[
  Z_n=\left\{(X,Y): P_n(X,Y)=0 \right\}. 
  \]
 We now determine the connected components (branches) of the intersection of  $Z_n$ with the first quadrant $\{(X,Y):X\geq 0, Y\geq 0\}$, see  Figure~\ref{fig:kus3}. 
  
\begin{lem}\label{lem:first branch}
 The intersection of  $Z_n$ with the first quadrant $\{(X,Y):X\geq 0, Y\geq 0\}$ consists of exactly $n$ connected components, all disjoint, and for each $m=1,\dots,n$ there is a unique component $Z_{n,m}$ linking the point $(0,m^2)$ to $(0,\lambda_{n,m})$. 
  \end{lem}
  \begin{figure}[ht]
\begin{center}
\includegraphics[height=60mm]{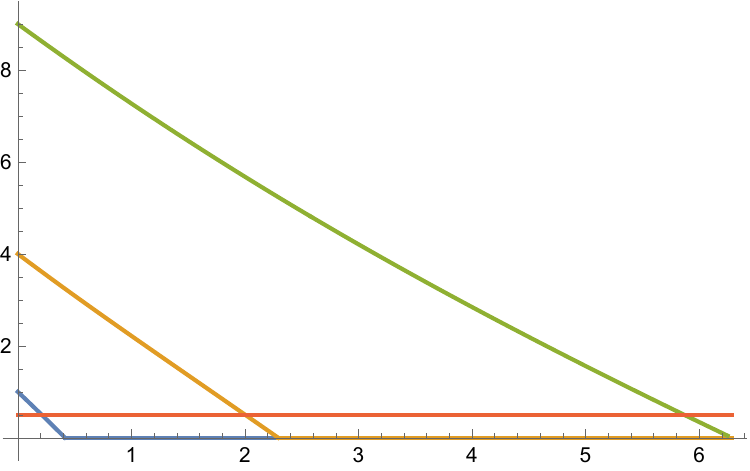}
\caption{  The zero locus $P_3(x,y)=0$. There are three branches, the $m$-th connecting the point $(m^2,0)$ to $(0,\lambda_{3,m})$ where  $0<\lambda_{3,1}<\lambda_{3,2} <\lambda_{3,3}$ are the zeros of the Laguerre polynomial $L_3(x)$.   The line $y=0.5$ intersects the zero locus in three points, all in the positive quadrant. }
\label{fig:kus3}
\end{center}
\end{figure}
\begin{proof}
For each fixed $y_0\geq 0$, we know from Lemma~\ref{lem:distinct roots of Q} that the intersection of $Z_n$ with the horizontal line $y=y_0$ consists of $n$ distinct points $(\alpha_m(y_0),y_0)$ with  $\alpha_1(y_0)<\dots <\alpha_n(y_0)$, of which exactly $n-\lfloor \sqrt{y_0}\rfloor$ lie in the first quadrant. 

 By Lemma~\ref{lem:Intersect with Y axis}, the intersection of $Z_n$  with the $y$-axis are the points $\{(0,m^2):m=1,\dots, n\}$, and by Lemma~\ref{prop:Laguerre}, the intersection with the $x$-axis are the points $(0,\lambda_{n,m})$, $m=1,\dots,n$ where $0<\lambda_{n,1}<\dots<\lambda_{n,n}$ are the zeros of the Laguerre polynomial $L_n(x)$.

For each $m=1,\dots,n$, the curve $(y,\alpha_m(y_0))$ lies in $Z_n$. When $y_0=0$, the points are $(y_m(0),0) = (\lambda_{n,m},0)$. 
 Since the points $\alpha_i(y_0)$ are distinct, we see that these curves are disjoint. Since $Z_n$ intersects the $y$-axis at the points $(0,m^2)$, $m=1,\dots,n$, we see that there is exactly one branch linking $(0,m^2)$ to the point $(\lambda_{n,m},0)$ and this is precisely
 \[
 Z_{n,m}=\{(\alpha_m(y),y):0\leq y\leq m^2\}.
 \]
\end{proof}

\subsection{Proof of Theorem~\ref{thm:Juddv1m}}

\begin{proof}
 We need to produce an infinite sequence of $N$'s, and points in the intersection $Z_m\cap Z_N$, which are distinct for different $N$'s.

Set $\lambda:=\lambda_{m,1}$, and let $Z_{m,1}$ be the unique branch of $Z_m$ linking $(0,1)$ to $(\lambda_{m,1},0)=(\lambda,0)$. For $N>m$, consider the branches $Z_{N,i}$ for all $i\geq 2$ for which $\lambda_{N,i}<\lambda $. By Gawronski's theorem (Theorem~\ref{Gawronski thm}), for $N\gg 1$, there are about $\frac  2\pi \sqrt{N \lambda }$ such branches. The intersection of such a branch $Z_{N,i}$  with the $y$-axis is  $(0,i^2)$, which lies above the intersection $(0,1)$ of $Z_{m,1}$ with the $y$-axis, while its intersection $(\lambda_{N,i},0)$ with the $x$-axis lies to the left of the intersection $(\lambda ,0)$ of $Z_{m,1}$ with the $x$-axis. Hence $Z_{N,i}$ must intersect $Z_{m,1}$, at some point in the first quadrant, as in Figure~\ref{figsub:Z20magnified}.

  \begin{figure}
 \begin{subfigure}[b]{0.45\textwidth}
    \includegraphics[height=50mm , width=\textwidth]{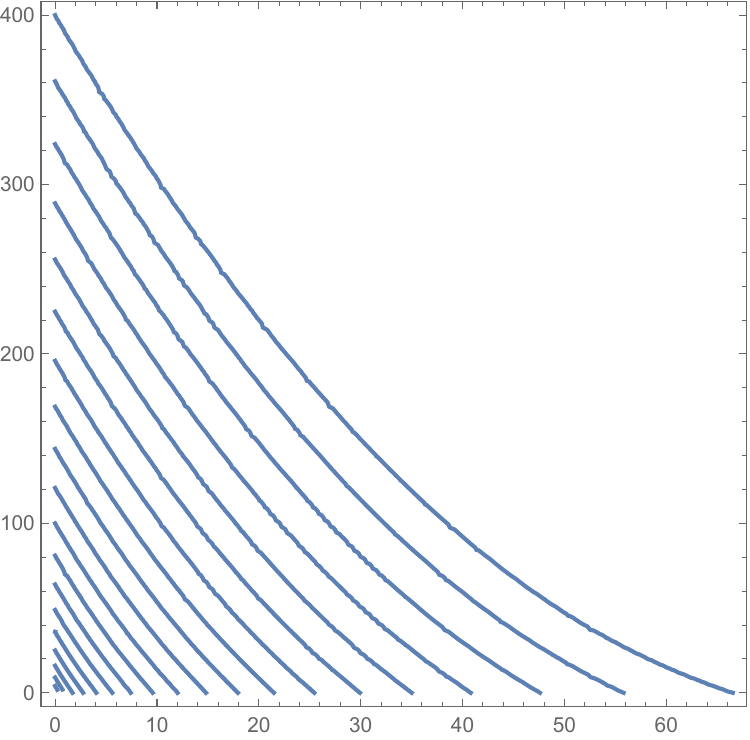}
    \caption{    }\label{figsub:Z20full}
   \end{subfigure}
    \hfill
   \begin{subfigure}[b]{0.45\textwidth}
    \includegraphics[height=50mm, width=\textwidth ]{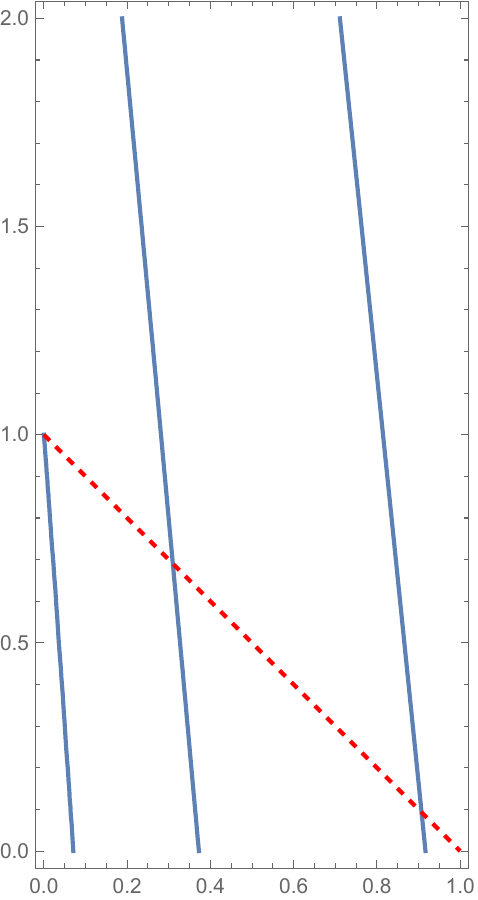}
    \caption{   }   \label{figsub:Z20magnified}
  \end{subfigure}
   \caption{The zero set $Z_{20}$ of the constraint polynomial $P_{20}(x,y)$ (left). The intercepts with the $y$-axis are at $y=m^2$, $m=1,\dots,20$, and the intercepts with the $x$-axis are at the zeros of the Laguerre polynomial $L_{20}(x)$, which are  $\lambda_{20,1}=0.0705399$, $\lambda_{20,2}=0.372127$, $\lambda_{20,3}=0.916582$, $\ldots, \lambda_{20,20} = 66.5244$.
    On the right, the region $0<x<1$, $0<y<2$, with   $Z_1=\{x+y-1=0\}$  (dashed, red).  }
     \label{fig:Z20}
\end{figure}

 Denote by $P_{N,k}\in Z_{m,1} \cap Z_{N,k}$ the first such intersection point (in case there is more than one), and by $\mathcal P(N)=\{P_{N,i}: \lambda_{N,i}\leq \lambda_{m,1}\}$.  
 The number of these points is hence
 \[
 \#\mathcal P(N) \sim \frac{2}{\pi}\sqrt{\lambda  N}.
 \]
    These points are all distinct because they belong to disjoint branches of $Z_N$.

 Take the sequence   $N_i=100^i$. Then the number of intersection points $\#\mathcal P(N_i)$    is $\sim  \frac {2}{\pi} \sqrt{\lambda N_i} $, while the number of those produced in the previous stages is 
 \begin{multline*}
\sum_{j\leq i-1} \#\mathcal P(N_i) \sim  \sum_{1\leq j\leq i-1}\frac  2\pi  \sqrt{\lambda N_j} = \frac 2\pi \sqrt{\lambda}  \sum_{1\leq j\leq i-1} 10^j
  \\
 <\frac {2\sqrt{\lambda}}{\pi  } \frac{ 10^i}{10-1} =\frac 1{9} \cdot \frac {2}{\pi} \sqrt{\lambda N_i} ,
 \end{multline*}
 which is strictly smaller than the number $\#\mathcal P(N_i)$ of new ones produced. Therefore we find a new point in
 $\mathcal P(N_i)$ on $Z_{m,1}\cap Z_{N_i}$, distinct from the earlier points $\mathcal P(N_j)\subset Z_{m,1}\cap Z_{N_j}$ for $j<i$. 
\end{proof}


\begin{thebibliography}{99}
\bibitem{Braak}
Braak, Daniel. Integrability of the Rabi Model.
Phys. Rev. Lett. 107, 100401 (2011).

\bibitem{Gawronski} 
Gawronski, Wolfgang.
On the asymptotic distribution of the zeros of Hermite, Laguerre, and {J}onqui\`ere polynomials.
J. Approx. Theory 50 (1987), no.3, 214--231.

\bibitem{HJ}
Horn, Roger A. and Johnson, Charles R.
Matrix analysis.
Cambridge University Press, Cambridge, 2013, xviii+643 pp.

\bibitem{KRBW}
Kimoto, Kazufumi; Reyes-Bustos, Cid; Wakayama, Masato.  Determinant expressions of constraint polynomials and the spectrum of the asymmetric quantum Rabi model.
Int. Math. Res. Not. IMRN (2021), no. 12, 9458--9544.

\bibitem{Kus}
Ku\'{s}, Marek.
On the spectrum of a two-level system. J. Math. Phys. 26 (1985), no.11, 2792--2795.

\bibitem{Larson&Mavrogordatos}
Larson, Jonas and Mavrogordatos, Themistoklis.
The Jaynes-Cummings Model and Its Descendants: Modern research directions.
 IOP Publishing, (2021). \url{https://dx.doi.org/10.1088/978-0-7503-3447-1}


\bibitem{LFTB}
Li, Zi-Min ; Ferri, Devid ; Tilbrook, David and  Batchelor, Murray T.
Generalized adiabatic approximation to the asymmetric quantum Rabi model: conical intersections and geometric phases. J. Phys. A 54 (2021), no.40, Paper No. 405201, 14 pp.




\bibitem{RBW}
Reyes-Bustos, Cid and Wakayama, Masato.
Degeneracy and hidden symmetry for the asymmetric quantum Rabi model with integral bias. Commun. Number Theory Phys. 16 (2022), no.3, 615--672.











\end{thebibliography}
\end{document}